\documentclass{article}
\usepackage[utf8]{inputenc}
\usepackage{amsmath}
\usepackage{amsfonts}
\usepackage{amsthm}
\usepackage{comment}
\usepackage{tikz-cd}
\usepackage{tikz}
\usepackage{cite}
\usepackage{authblk}
 \usepackage[margin=1.5in]{geometry}

\usetikzlibrary{shapes,positioning,intersections,quotes}

\DeclareMathOperator{\Id}{Id}
\numberwithin{equation}{section}

\newtheorem{theorem}{Theorem}[section]
\newtheorem{definition}{Definition}[section]
\newtheorem{lemma}{Lemma}[section]
\newtheorem{remark}{Remark}[section]
\newcommand{\R}{\mathbb{R}}

\newcommand{\der}{\mathrm{d}}
\title{Unisolvent and minimal physical degrees of freedom for the second family of polynomial differential forms}
\author[$\dagger$]{Ludovico Bruni Bruno\thanks{email: ludovico.brunibruno@unitn.it}}
\author[$\dagger$]{Enrico Zampa\thanks{email: enrico.zampa@unitn.it}}

\affil[$\dagger$]{Dipartimento di Matematica, Università di Trento, via Sommarive 14, 38123, Povo (TN), Italia} 
\date{May 2022}

\begin{document}
\maketitle

\begin{abstract}
The principal aim of this work is to provide a family of unisolvent and minimal physical degrees of freedom, called weights, for N\'ed\'elec second family of finite elements. Such elements are thought of as differential forms $ \mathcal{P}_r \Lambda^k (T)$ whose coefficients are polynomials of degree $ r $. We confine ourselves in the two dimensional case $ \R^2 $ since it is easy to visualise and offers a neat and elegant treatment; however, we present techniques that can be extended to $ n > 2 $ with some adjustments of technical details. In particular, we use techniques of homological algebra to obtain degrees of freedom for the whole diagram
$$ \mathcal{P}_r \Lambda^0 (T) \rightarrow \mathcal{P}_r \Lambda^1 (T) \rightarrow \mathcal{P}_r \Lambda^2 (T), $$
being $ T $ a $2$-simplex of $ \R^2 $. This work pairs its companions recently appeared for N\'ed\'elec first family of finite elements.
\end{abstract}

\section{Introduction} 
Degrees of freedom are one of the main ingredients of a finite element triple as defined by Ciarlet \cite{Ciarlet}. For standard polynomial Lagrange elements over simplices, the classical degrees of freedom are evaluations on the principal lattice $L_r(T)$ of top-dimensional simplices $T$ of the triangulation. These degrees of freedom have a clear physical meaning: if $u_h$ is the numerical solution, then degrees of freedom are just the values of the exact solution at some points of the mesh. On the other side, for the polynomial differential forms families $\mathcal{P}_r^-\Lambda^k$ and $\mathcal{P}_r\Lambda^k$ described in \cite{FEEC}, the standard degrees of freedom are the so called \emph{moments}, that is, integrals against $(d-k)$-forms on $d$-subsimplices, for $d = k,\ldots,n$, where $n = \dim T $ is the dimension of the domain of the problem. These degrees of freedom have some disadvantages, which we aim here to improve:
\begin{enumerate}
\item they lack an immediate physical interpretation;
\item the associated Vandermonde matrix is not well conditioned; 
\item they are difficult to implement. 
\end{enumerate}
To overcome these issues, another choice of degrees of freedom has been proposed in \cite{BossBook}. It consists in considering integrals over $k$-cells topologically contained in the top dimensional simplices. These degrees of freedom are called \emph{weights} or \emph{physical}, since they have a clear physical interpretation: circulations or fluxes for vector fields ($1$- and $n-1$-forms) and averages for densities ($n$-forms). Moreover, weights are a straightforward generalization of the evaluation-type degrees of freedom for scalar functions (for $k=0$, a $k$-cell is just a point and the integral is just the evaluation). 
Physical degrees of freedom for the first (or \emph{trimmed}) family $\mathcal{P}_r^-\Lambda^k$, whose features in the framework that we adopt here have been pointed out in several works, such as \cite{Gopa}, \cite{Hipt} and \cite{FEEC}, were studied extensively in \cite{RapettiBossavit}, \cite{ChristiansenRapetti} and more recently in \cite{ENUMATH} and \cite{nostropreprint}. A slightly different point of view is also offered in \cite{Lohi1} and \cite{Lohi2}. On the other side, for the second (or \emph{complete}) family $\mathcal{P}_r\Lambda^k$ the first physical degrees of freedom for the two dimensional case were proposed in \cite{ZAR} only recently, where however unisolvence was not proved, but only checked numerically. In this work, we stick to the two dimensional case and we provide different physical degrees of freedom for the second family and we rigorously prove the unisolvence using cohomological tools. Moreover, we provide numerical evidence of the well-conditioning of the associated Vandermonde matrix and we perform some interpolation tests.

We assume that the reader is familiar with standard notions in differential geometry and algebraic topology that are now common in most works on Finite Element Exterior Calculus and Finite Element Systems, such as differential forms, differential complexes, cellular complexes, chains, cochains, cohomology, de Rham maps, and so on. See section 2 of \cite{christiansen2009Foundations} for a concise introduction on these topics.  We however recall known and useful facts when setting the notation.

The outline of this work is as follows. In section \ref{sect:prel} we introduce basic definitions and tools. We recall known results and state Lemmas that we will use in the subsequent. In section \ref{sect:mainsect} we state the main results concerning the construction of unisolvent and minimal sets. In particular, confining ourselves in the case of $ \mathbb{R}^2 $, we identify a unisolvent and minimal sequence for N\'ed\'elec second family. In section \ref{sect:NumTest} we present some numerical results concerning the generalised Vandermonde matrices associated with the introduced families and the associated interpolators, comparing an example of convergence of a smooth, oscillating form with that of its differential. We summarise conclusions and propose future developments in section \ref{sect:concl}.

\section{Physical systems of degrees of freedom} \label{sect:prel}

In this section we recall the definition of a physical system of degrees of freedom and some results from \cite{ZAR}. 

Let $X$ be a compatible finite element system in the sense of Christiansen \cite{ChristiansenConstruction}, \cite{ChristiansenTopics}, \cite{ChristiansenRapetti} over the cellular complex $\mathcal{T}$. In particular, for each $T$ in $\mathcal{T}$ (of any dimension\footnote{In the finite element systems framework, one considers spaces of differential forms and degrees of freedom on cells of each dimensions, not only on top dimensional ones.}), for each $k = 0,1,2,\ldots, \dim T$, the following sequence is exact: 
\begin{equation*}
    0 \rightarrow \R \hookrightarrow X^0(T) \overset{\der}{\rightarrow} X^1(T) \overset{\der}{\rightarrow} \dots \overset{d}{\rightarrow} X^{\dim T}(T) \rightarrow 0.  
\end{equation*}
Here the first arrow is the inclusion and $\mathrm{d}$ denotes the exterior derivative. 
Moreover, for $T$ in $\mathcal{T}$ we denote with $\mathring{X}^k(T)$ the subspace of $X^k(T)$ made of all forms with zero trace on the boundary. 
\begin{definition}
A \emph{system of physical degrees of freedom} (physical sysdofs) $\mathcal{F}$ over $X$ is a choice, for each cell $T$ in $\mathcal{T}$ , for each $k = 0,1,2,\ldots, \dim T$, of a finite set $\mathring{\mathcal{F}}^k(T)\doteq \{s_1,\ldots,s_{\mathring{N}^k(T)}\}$ of non-overlapping $k$-cells. These cells induce functionals 
\begin{equation}
    \omega\mapsto w(\omega,s_i)\doteq \int_{s_i}\omega.
\end{equation}
We call $w(\omega,s_i)$ the \emph{weight} of $\omega$ on $s_i$. 
\end{definition}
The unisolvence of a physical system of degrees of freedom is defined in the obvious way. \begin{definition} 
A physical sysdofs is said to be \emph{unisolvent} if, for each $T$ in $\mathcal{T}$, for each $k=0,1,2,\ldots, \dim T$, the only form $\omega$ in $\mathring{X}^k(T)$
which satisfies 
\[ w(\omega,s) = 0, \quad\forall s\in\mathring{\mathcal{F}}^k(T) \]
is the zero form. 
\end{definition}
Clearly, a unisolvent physical sysdofs must satisfy the trivial necessary condition: for each $T$ in $\mathcal{T}$, for each $k = 0,1,2,\ldots$, $\mathring{N}^k(T) \geq \dim\mathring{X}^k(T)$ where $\mathring{N}^k(T)$ denotes the cardinality of the set $\mathring{\mathcal{F}}^k(T)$. This motivates the following definition.
\begin{definition}
A physical sysdofs is \emph{minimal} if, for each $T$ in $\mathcal{T}$, for each $k = 0,1,2,\ldots, \dim T$, the following equality holds:
\[ N^k(T) = \dim \mathring{X}^k(T). \] 
\end{definition}

From the properties of compatible finite element system we obtain the following equivalent definition of unisolvence and minimality, which is closer to classical one found in standard books on finite elements \cite{Ciarlet}. For each $S,T$ in $\mathcal{T}$ we write $S\leq T$ is $S$ is a subcell of $T$. Moreover, for $T$ in $\mathcal{T}$ and $k = 0,1,2,\ldots, \dim T$, write
\begin{equation}
    \mathcal{F}^k(T) \doteq \bigcup_{S\leq T}\mathring{\mathcal{F}}^k(S). 
\end{equation}
\begin{lemma} 
If a physical system of degrees of freedom $\mathcal{F}$ is unisolvent, then, for each top dimensional cell $T$ in $\mathcal{T}$, for each $k = 0,1,2,\ldots, \dim T$, the only form $\omega$ in $X^k(T)$ satisfying 
\begin{equation}
    w(\omega,s) = 0,\qquad \forall s\in\mathcal{F}^k(T) 
    \label{eq: alternative unisolvence}
\end{equation} 
is the zero form. Moreover $\mathcal{F}$ is minimal and unisolvent if and only if the above condition holds and  $N^k(T) = \dim X^k(T)$, where $N^k(T)$ denotes the cardinality of $\mathcal{F}^k(T)$.
\end{lemma}
\begin{proof}
Assume that $\mathcal{F}$ is unisolvent. Let $\omega\in X^k(T)$ satisfying condition \eqref{eq: alternative unisolvence}. Then let $S$ any $k$-subcell of $T$ and let $\iota_{S,T}:S\to T$. Clearly $\iota_{S,T}^*\omega$ belongs to $X^k(S)$, but since it is a $k$-form on a $k$-cell, its traces on the boundary of $S$ vanish by definition, therefore $\iota_{S,T}^*\omega$ actually belongs to $\mathring{X}^k(S)$. Then, by unisolvence, $\iota_{S,T}^*\omega = 0$. Let now $S$ be a $k+1$-subcell of $T$ and $S'$ be a $k$-cell belonging to the boundary of $S$. Then $\iota_{S',S}^*\iota_{S,T}^*\omega = \iota_{S',T}^*\omega = 0$ by the previous argument. Therefore $\iota_{S,T}^*\omega$ belongs to $\mathring{X}^k(S)$. Again, by unisolvence, $\iota_{S,T}^*\omega = 0$. Proceeding in this way, we obtain that $\omega\in\mathring{X}^k(T)$. Finally, unisolvence gives $\omega = 0$. For the stronger statement, see Proposition 2.5 in \cite{ChristiansenRapetti}.
\end{proof}

From the computational point of view, one may check equivalent an condition for any given top dimensional cell $T$ and any $ k = 0,1,2,\ldots $. We thus define the \emph{generalized Vandermonde matrix} $ V $, whose $(i,j)$-th element is
$$ \int_{s_i} \omega_j ,$$
being $ \omega_1, \ldots, \omega_{N^k(T)} $ some basis for $X^k (T) $. We thus have the following.

\begin{lemma}
A collection of $ k $-cells $ \{ s_1, \ldots, s_{N^k}(T) \} $ is unisolvent and minimal if and only if $V$ is a square full rank matrix. Such a rank does not depend on the basis $ \{ \omega_1, \ldots, \omega_{\dim X^k(T)} \}$ chosen for $ X^k(T) $.
\end{lemma}

\subsection{A motivation: the scalar case}
To fix ideas, let $ T $ be a $2$-simplex, i.e. a non degenerate triangle. Notice that, for $ k = 0 $ and $X^0(T) = \mathcal{P}_r(T)$, the problem of deducing unisolvence and minimality is linked to the problem of deducing if a collection of nodes $ \mathcal{N} $ in $ \R^2 $ is \emph{poised}, which means that the only polynomial vanishing on $ \mathcal{N} $ is the zero polynomial. Explicitly, for a polynomial $ \varphi \in \mathcal{P}_r (\R^2) $ this reads as
$$ \varphi (\pmb{x}) = 0 \quad \forall \pmb{x} \in \mathcal{N} \ \Longrightarrow \varphi (\pmb{x}) = 0 \quad \forall \pmb{x} \in \R^2 .$$
This problem is still unsolved in its greatest generality, however several partial results and conjectures have been offered. A possible approach to a complete understanding of the placement of points in $ \R^2 $ consists in studying the number of lines that pass through a fixed number of points of $ \mathcal{N} $. This does not give \emph{all possible unisolvent sets}, but the conjectural result claims these collections are \emph{all unisolvent}, see \cite{GCGM}. This approach is convenient in this framework, since when considering particular collection of points, such as \emph{principal lattices} or \emph{regular lattices} \cite{ChungYao} and some of their subsets, one may reduce the problem. 

These considerations clearly also extend to greater $ k $, in this context to $ k = 1 $ (that is, to edges) and $ k = 2 $ (that is, to faces). Some numerical results relate these two problems. In particular, for $ k = 1 $ we address the reader to \cite{nostropreprint} and for $ k = 2 $ to \cite{ABRICO}.

\subsection{Interpolators and (co)-homological tools} \label{sect:interpdef}
For each $T$ in $\mathcal{T}$, for each $k = 0,1,2,\ldots, \dim T$, a physical sysdofs induces an interpolator $\Pi^k(T) :\Lambda^k(T) \to X^k(T)$ by the equations: 
\begin{equation}
   w(\omega,s) = w(\Pi^k(T)\omega,s),\qquad\forall s\in\mathcal{F}^k(T). 
   \label{eq:interpolator}
\end{equation}
The interpolator is well defined if the physical sysdofs is unisolvent. In fact, assume that $\Pi^k\omega$ and $\tilde{\Pi}^k\omega$ are two interpolators which satisfy \eqref{eq:interpolator}. Then 
\[ w (\Pi^k\omega - \tilde{\Pi}^k\omega)  = 0, \]
and unisolvence gives $\omega = 0$. 

We are interested in interpolators that commute with the exterior derivative, that is, such that the following diagram is commutative
\begin{equation*}
    \begin{tikzcd}
\Lambda^k(T) \arrow[r, "\der"] \arrow[d, "\Pi^k(T)"] & \Lambda^{k+1}(T) \arrow[d, "\Pi^{k+1}(T)"] \\
X^k(T) \arrow[r, "\der"]                             & X^{k+1}(T)                                
\end{tikzcd}
\end{equation*}
In \cite{ZAR} Zampa et al. showed that an interpolator induced by a physical sysdofs commutes with the exterior derivative if and only if the union 
\begin{equation}
    \mathcal{F}^{\bullet}(T)\doteq \bigcup_{k = 0}^{\dim T}\mathcal{F}^k(T) = \bigcup_{k = 0}^{\dim T}\bigcup_{S\leq T}\mathring{\mathcal{F}}^k(S)
\end{equation}
is a cellular complex, that is, if and only if the boundary of a cell in $\mathcal{F}^{k+1}(T)$ is a union of cells in $\mathcal{F}^k(T)$. 

If this is the case, we can consider $k$-chains $C_k(\mathcal{F}^{\bullet}(T))$ and $k$-cochains $C^k(\mathcal{F}^{\bullet}(T))$ over $\R$. Denote with $\delta$ the coboundary operator mapping $k$-cochains to $k+1$-cochains \cite{Hatcher}. It is natural then to consider the de Rham map \cite{Whitney}
\begin{equation}
\begin{split}
    \mathfrak{R}^k: X^k(T) & \to C^k(\mathcal{F}^{\bullet}(T)) \\
    \omega & \mapsto \left( c\mapsto \int_c\omega \right).
    \end{split}
    \label{eq: de Rham map}
\end{equation}
Stokes Theorem \cite{Lee} implies that the de Rham map commutes with the exterior derivative, that is, is a chain map. We can then arrange everything in a commutative diagram
\begin{equation}
\small
   \begin{tikzcd}
0 \arrow[r] & \R \arrow[r, hook] \arrow[d, "\Id"] & X^0(T) \arrow[r, "\der"] \arrow[d, "\mathfrak{R}^0"]  & X^1(T) \arrow[r, "\der"] \arrow[d, "\mathfrak{R}^1"] & \dots \arrow[r, "\der"]      & X^{\dim T}(T) \arrow[r] \arrow[d, "\mathcal{R}^{\dim T}"] & 0 \\
0 \arrow[r] & \R \arrow[r, "\psi"]                & C^0(\mathcal{F}^{\bullet}(T))) \arrow[r, "\delta"] & C^1(\mathcal{F}^{\bullet}(T)) \arrow[r, "\delta"] & \dots \arrow[r, "\delta"] & C^{\dim T}(\mathcal{F}^{\bullet}(T)) \arrow[r]           & 0
\end{tikzcd}
\label{eq: diagram}
\end{equation}
where $\psi$ is the unique map that makes it commutative, sending $1$ to the $0$-cochain $c\mapsto 1$. Notice that the top sequence is exact since $X$ is a compatible finite element system. 
We can thus give an equivalent characterization of unisolvence and minimality in terms of the de Rham map. 
\begin{lemma}
A physical sysdofs $\mathcal{F}$ is unisolvent (unisolvent and minimal) if and only if, for each $T$ in $\mathcal{T}$ and for each $ k $ the de Rham map \eqref{eq: de Rham map} is injective (an isomorphism of vector spaces). 
\end{lemma}
In \cite{ZAR}, the authors showed that a unisolvent and minimal physical system of degrees of freedom that induces commuting interpolators  must satisfy the following condition: the union of all cells in $\mathcal{F}^\bullet(T)$ paves $T$. If this this is the case, the bottom sequence in \eqref{eq: diagram} is exact.

\section{Physical degrees of freedom for the second family} \label{sect:mainsect}
In this section we will construct a physical sysdofs for the finite element system $X^k(T) = \mathcal{P}_{r-k}\Lambda^k(T)$ with $r\geq 2$ in the two-dimensional case. Since the majority of the following results are general, we claim and prove them in the case of an $n$-simplex $ T $; the specific case of interest here is immediately obtained for $ n = 2 $. We will exploit features of $ \mathbb{R}^2 $ only for the definition of $ \mathcal{F} $ and hence in Theorem \ref{thm:main}. We invite the reader to match the following construction with that for N\'ed\'elec first family \cite{NedelecFirst} given in \cite{ENUMATH}. Recall that spaces $ \mathcal{P}_{r-k}\Lambda^k(T) $ are defined as subspaces of differential $k$-forms $ \Lambda^k (T) $ whose coefficient are polynomials of degree $ \leq r-k $. These spaces are sometimes called \emph{complete}, since they are precisely tensor products $ \mathcal{P}_{r-k} (T) \otimes \mathrm{Alt}^k(T) $, being $ \mathcal{P}_{r-k} (T) $ the space of polynomials of degree at most $ r-k $ in $n$ variables defined on $ T $ and $ \mathrm{Alt}^k(T) $ that of linear alternating $k$-forms on (the tangent bundle of) $ T $. This makes it easy the computation of 
$$ \dim \mathcal{P}_r \Lambda^k (T) = \dim \mathcal{P}_r (T) \cdot \dim \mathrm{Alt}^k(T) = \binom{r + \dim T}{\dim T} \binom{\dim T}{k}.$$
When $ n = 2 $, proxies of this sequence are known as N\'ed\'elec second family \cite{NedelecSecond} and the central space is that of \cite{BDM}. Note that we use the subscript $r-k$ instead of the classical $r$ found in the literature, since the exterior derivative lowers the polynomial degree at each stage of the complex
\[ \mathcal{P}_r\Lambda^0(T) \overset{\der}{\rightarrow} \mathcal{P}_{r-1}\Lambda^1(T) \overset{\der}{\rightarrow} \dots\overset{\der}{\rightarrow} \mathcal{P}_{r-\dim T}\Lambda^{\dim T}(T). \]
We recall now the definition of \emph{small simplex} from \cite{ChristiansenRapetti}. For $ n = \dim T $ and $r \geq 0 $ let $\mathcal{I}(r,n)$ be the set of multi-indices $\pmb{\alpha} = (\alpha_0,\ldots,\alpha_n)$ with nonnegative components and such that $|\boldsymbol{\alpha} | \doteq \alpha_0 + \ldots + \alpha_n = r$. If $T$ is a simplex of dimension $n$ and vertices $\{\pmb{x}_0,\ldots,\pmb{x}_n\}$, we equip it with barycentric coordinates $\{\lambda_0,\ldots,\lambda_n \}$, i.e. the only (up to permutations) non negative degree $1$ polynomials defined on $ T $ such that 
$$ \boldsymbol{x} = \sum_{i=0}^n \lambda_i \boldsymbol{x}_i , \qquad \sum_{i=0}^n \lambda_i = 1 , \qquad \forall \boldsymbol{x} \in T.$$
For each $\pmb{\alpha}\in\mathcal{I}(r-1,n)$ we define the \emph{small} $n$-simplex $s^{\pmb{\alpha}}$ as the image of $T$ under the homothety 
\begin{equation}
z_{\pmb{\alpha}} \ : \quad \pmb{x} \ \mapsto \ z_{\pmb{\alpha}}(\pmb{x}) = \frac{1}{r} \sum_{i=0}^n [\lambda_i(\pmb{x}) + \alpha_i]\, \pmb{x}_i\,.
\label{eq: homothety}
\end{equation}
Note that \eqref{eq: homothety} is just the identity for $r = 1$. Small $k$-simplices are just $k$-subsimplices of small $n$-simplices and we denote them with $\Sigma^k_r(T)$. In particular $\Sigma^0_r(T)$ is the principal lattice $L_r(T)$, that is, the set of points with barycentric coordinates 
\[ \Sigma_r^0 (T) \doteq \frac{1}{r}(\lambda_0 + \alpha_0, \ldots, \lambda_n + \alpha_n ),\qquad \pmb{\alpha}\in\mathcal{I}(r,n). \]
 If the reader is familiar with weights for N\'ed\'elec first family they might have noted that a slightly different definition of small simplices is usually provided. In particular, the term $ \lambda_i (\boldsymbol{x}) $ in \eqref{eq: homothety} is usually omitted, so that overlappings are avoided. We shall see the reason of such a different choice in the subsequent of this section.
 
For $\pmb{\xi}\in T$, define the affine tranformation 
\begin{equation}
    \tau_{\pmb{\xi}} \ : \quad\pmb{x}\mapsto \lambda_0(\pmb{\xi})\pmb{x} + \sum_{i=1}^n\lambda_i(\pmb{\xi})\pmb{x}_i. \label{eq: tau map}
\end{equation}
Note that the map \eqref{eq: tau map} is invertible if and only if $\lambda_0(\pmb{\xi}) \neq 0$. We define $ T_{\boldsymbol{\xi}} \doteq \tau_{\boldsymbol{\xi}} (T) $ and let $ \tau^*_{\boldsymbol{\xi}} $ denote the pullback with respect to $ \tau_{\boldsymbol{\xi}} $. We have the following.
\begin{lemma} \label{lem:useful}
Let $ \omega \in \mathcal{P}_{r-\dim T} \Lambda^{\dim T} (T)$ be such that 
$$ \int_{T} \tau_{\boldsymbol{\xi}} \omega = 0, \qquad \forall \boldsymbol{\xi} \in \mathbb{R}^{\dim T}.$$
Then $ \omega = 0 $.
\end{lemma}

\begin{proof}
This is a direct consequence of \cite[Lemma $3.12$]{ChristiansenRapetti}.
\end{proof}

\begin{theorem}
Let $\Gamma = \{\pmb{\xi}_1,\ldots,\pmb{\xi}_{N^2(T)}\}$ be a poised subset of $L_{r}(T)$ such that $\lambda_0(\pmb{\xi}_i) > 0$ for $i=1,\ldots,N^2(T)$. Let $ \omega \in \mathcal{P}_{r-\dim T} \Lambda^{\dim T} (T)$ be such that  \[ \int_{\tau_{\pmb{\xi}}(T)}\omega= 0,\qquad \forall \pmb{\xi}\in \Gamma. \]
Then $\omega = 0$. 
\end{theorem}
\begin{proof}
The map 
\[ \pmb{\xi}_i \mapsto\int_{\tau_{\pmb{\xi}_i}(T)}\omega \]
is a polynomial of degree $r$ in $\dim T$ variables $ \xi_1, \ldots, \xi_{\dim T} $ which vanishes on $\lvert L_r(T)\rvert$ points of a poised set, therefore is zero for each $ \boldsymbol{\xi} \in \mathbb{R}^{\dim (T)}$. It follows from Lemma \ref{lem:useful} that $\omega  = 0$. 
\end{proof}

As an example of set $ \Gamma $, we may pick any set satisfying the GC condition \cite{GCCarni} (see also \cite{GCGeneral} and \cite{deBoor} for higher dimensional counterparts). Some explicit examples can be found in \cite{GCr4} and \cite{GCr5} and we offer more in a recursive fashion in the following.

We define $\mathcal{F}$ as follows. Let $T$ be a $2$-simplex. For $k = 0$, $\mathcal{F}^0(T)$ is just the principal lattice $L_r(T)$. 
For $k = 2$ we consider the GC set $\Gamma_r = \{\pmb{\xi}_1,\ldots,\pmb{\xi}_{N^2(T)}\}$, which is a subset of $L_r(T)$ of cardinality $N^2(T) = \dim\mathcal{P}_{r-2}\Lambda^2(T) = \frac{r(r-1)}{2}$. For $i = 1,\ldots, N^2(T)$, define the subset 
\[ \Gamma_r(i) \doteq \{ \pmb{\xi} \in\ \Gamma \mid \lambda_0(\pmb{\xi}) < \lambda_0 (\pmb{\xi}_i) \} \]
We define $\mathcal{F}^2(T)$ as the set $\{s_1,\ldots,s_{N^2(T)}\}$ where 
\begin{equation}
    s_i = \overline{ \tau_{\pmb{\xi}_i}(T) \setminus \left(\bigcup_{\pmb{\xi}\in\Gamma_r(i)}\tau_{\pmb{\xi}}(T) \right) }.
    \label{eq : def case k = 2}
\end{equation}
The closure is needed to preserve the structure of cell. Finally, we define $\mathcal{F}^1(T)$ as the subset of $\Sigma_r^1(T)$ made of those small $1$-simplices that are on the boundary of cells in $\mathcal{F}^2(T)$.

We now propose a possible choice of $\Gamma_r$ for each polynomial degree $r$. We identify each point $x$ of $T$ with the triple $(\lambda_0 (x) , \lambda_1 (x) , \lambda_2( x))$ (e.g. the barycenter is $(1/3,1/3,1/3)$). Let 
\begin{equation}
    \Gamma_r = \begin{cases} \{ (1, 0, 0 ) \}  \text{ if $r = 2$, }\\ 
    \tau_{\boldsymbol{\zeta}_r}(\Gamma_{r-1}) \cup \Delta_r \text{ if $r > 2$ }, \end{cases}
\end{equation}
where 
\begin{align*}
    \boldsymbol{\zeta}_r & = \begin{cases} \left ( \frac{r -1}{r}, 0, \frac{1}{r} \right ) \text{ if $r$ is odd,}\\
    \left ( \frac{r -1}{r}, \frac{1}{r}, 0 \right ) \text{ if $r$ is even, } \end{cases} \\ 
    \Delta_r & = \begin{cases} \left\{\frac{1}{r} ( i, 1 - i, 0 ) \text{ for } i = 1,\ldots,r, \ i\neq \frac{r+1}{2} \right\}, \text{ if $r$ is odd,}\\ 
    \left\{\frac{1}{r} ( i, 0 , 1- i) \text{ for } i = 1,\ldots,r, \ i\neq \frac{r}{2} \right\},\text{ if $r$ is even. }\end{cases}
\end{align*}
For example $\Gamma_3$ is 
\begin{align*}
    \Gamma_3 &= \left \{ \left ( \frac{2}{3}, 0, \frac{1}{3} \right ) , \left ( \frac{1}{3}, \frac{2}{3}, 0 \right ) , ( 1, 0, 0 ) \right \},
\end{align*}
since $\tau_{( 2/3 , 0 , 1/3 )}$ maps $( 1, 0, 0)$ to $ ( 2/3 , 0 , 1/3 )$ and $\Delta_3 = \{ (1/3, 2/3, 0 ) , (1, 0, 0) \}$. Similarly, $\Gamma_4$ is given by 
\begin{equation*}
    \Gamma_4 = \left \{ \left ( \frac{1}{2}, \frac{1}{4}, \frac{1}{4} \right ), \left ( \frac{1}{4}, \frac{3}{4}, 0 \right ), \left ( \frac{3}{4}, \frac{1}{4}, 0 \right ), \left ( \frac{1}{4}, 0, \frac{3}{4} \right ), \left ( \frac{3}{4}, 0, \frac{1}{4}\right ), (1, 0, 0) \right \}.
\end{equation*}
See Figure \ref{fig:cells} for a depiction of the set $\Gamma_r$ and the resulting cells $\mathcal{F}$ for $r = 2$, $3$ and $4$.

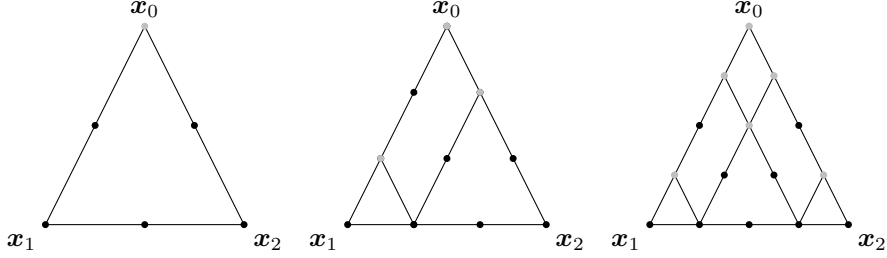
\begin{figure}
    \centering
	\begin{tikzpicture}[scale = .66]
	\draw (0,0) -- (4,0) -- (2,4) -- cycle;
	\fill (0,0) circle (2pt);
	\fill (4,0) circle (2pt);
	\fill (2,0) circle (2pt);
	\fill (1,2) circle (2pt);
	\fill (3,2) circle (2pt);
	
    \draw (0,0) node [anchor = north east] {$ \boldsymbol{x}_1 $};
	\draw (2,4) node [anchor = south] {$ \boldsymbol{x}_0 $};
	\draw (4,0) node [anchor = north west] {$ \boldsymbol{x}_2 $};
	
	\fill[lightgray] (2,4) circle (2pt);
\end{tikzpicture}
	\begin{tikzpicture}[scale = 0.44]
	\draw (0,0) -- (6,0) -- (3,6) -- cycle;
	\fill (0,0) circle (3pt);
	\fill (2,0) circle (3pt);
	\fill (4,0) circle (3pt);
	\fill (6,0) circle (3pt);
	\fill (2,4) circle (3pt);
	\fill (2,0) circle (3pt);
	\fill (1,2) circle (3pt);
	\fill (3,2) circle (3pt);
	\fill (3,6) circle (3pt);
	\fill (4,4) circle (3pt);
	\fill (5,2) circle (3pt);
	
	\draw (1,2) -- (2,0) -- (4,4);
	\draw (0,0) node [anchor = north east] {$ \boldsymbol{x}_1 $};
	\draw (3,6) node [anchor = south] {$ \boldsymbol{x}_0 $};
	\draw (6,0) node [anchor = north west] {$ \boldsymbol{x}_2 $};
	
	\fill[lightgray] (3,6) circle (3pt);
	\fill[lightgray] (4,4) circle (3pt);
	\fill[lightgray] (1,2) circle (3pt);
\end{tikzpicture}
	\begin{tikzpicture}[scale = .66]
	\draw (0,0) -- (4,0) -- (2,4) -- cycle;
	\fill (0,0) circle (2pt);
	\fill (4,0) circle (2pt);
	\fill[lightgray] (2,4) circle (2pt);
	\fill (2,0) circle (2pt);
	\fill (1,2) circle (2pt);
	\fill (3,2) circle (2pt);
	\fill (1,0) circle (2pt);
	\fill (3,0) circle (2pt);
	\fill[lightgray] (0.5, 1) circle (2pt);
	\fill (1.5,1) circle (2pt);
	\fill[lightgray] (1.5,3) circle (2pt);
    \fill[lightgray] (3.5,1) circle (2pt);
	\fill[lightgray] (2.5,3) circle (2pt);
	\fill (2.5,1) circle (2pt);
	\fill[lightgray] (2,2) circle (2pt);
	
	\draw (0.5, 1) -- (1,0) -- (2.5, 3);
	\draw (1.5, 3) -- (3,0) -- (3.5, 1);
	
	\draw (0,0) node [anchor = north east] {$ \boldsymbol{x}_1 $};
	\draw (2,4) node [anchor = south] {$ \boldsymbol{x}_0 $};
	\draw (4,0) node [anchor = north west] {$ \boldsymbol{x}_2 $};
	
	\fill[lightgray] (2,4) circle (2pt);
	\fill[lightgray] (0.5, 1) circle (2pt);
	\fill[lightgray] (1.5,3) circle (2pt);
    \fill[lightgray] (3.5,1) circle (2pt);
	\fill[lightgray] (2.5,3) circle (2pt);
	\fill[lightgray] (2,2) circle (2pt);
\end{tikzpicture}
\caption{Cells of $ \mathcal{F} $ for $ r = 2 $, $ r = 3 $ and $ r = 4 $, left to right. Gray dots represent the set $\Gamma_r$, that is, vertices of the triangles considered as (small) $ 2 $-simplices.} \label{fig:cells}
\end{figure}

\begin{remark}\label{rmk:hierar}
The recursiveness in the definition of $\Gamma_r$ gives a hierarchy on the weights associated with these cells. In fact, as degree $ r $ is increased by one, the associated family $ \mathcal{F} $ is obtained by adding a stripe on one side of the triangle, as shown in Figure \ref{fig:hierarchy}. 

\begin{figure}[h]
\centering
	\begin{tikzpicture}[scale = .75]
	\draw (0,0) -- (3,0) -- (1.5,3) -- cycle;
	\fill (0,0) circle (2pt);
	\fill (4,0) circle (2pt);
	\fill (2,4) circle (2pt);
	\fill (2,0) circle (2pt);
	\fill (1,2) circle (2pt);
	\fill (3,2) circle (2pt);
	\fill (1,0) circle (2pt);
	\fill (3,0) circle (2pt);
	\fill (0.5, 1) circle (2pt);
	\fill (1.5,1) circle (2pt);
	\fill (1.5,3) circle (2pt);
    \fill (3.5,1) circle (2pt);
	\fill (2.5,3) circle (2pt);
	\fill (2.5,1) circle (2pt);
	\fill (2,2) circle (2pt);
	
	\draw (0.5, 1) -- (1,0);
	\draw (1,0) -- (2,2);
	\draw[dotted] (3,0) -- (3.5, 1);
	\draw[dotted] (3,0) -- (4,0) -- (2,4) -- (1.5,3);
	\draw[dotted] (3,0) -- (3.5,1);
    \draw[dotted] (2,2) -- (2.5,3);
    \draw[dashed] (0,0) -- (-1,0) -- (1.5, 5) -- (2,4);
    \draw[dashed] (0.5,1) -- (0,2);
    \draw[dashed] (1.5,3) -- (1,4);
    \draw[dashed] (0,0) -- (-0.5,1);
	
	\fill (-1, 0) circle (2pt);
	\fill (-0.5,1) circle (2pt);
	\fill (0, 2) circle (2pt);
	\fill (0.5, 3) circle (2pt);
	\fill (1, 4) circle (2pt);
	\fill (1.5,5) circle (2pt);
	
	\draw (-1,0) node [anchor = north east] {$ \boldsymbol{x}_1 $};
	\draw (1.5,5) node [anchor = south] {$ \boldsymbol{x}_0 $};
	\draw (4,0) node [anchor = north west] {$ \boldsymbol{x}_2 $};
	
\end{tikzpicture}
\caption{Cells of $ \mathcal{F} $ for $ r = 5 $. Step from $ r = 3$ to $ r=4 $ are obtained by adding the dotted part, step from $ r =4$ to $ r = 5$ is obtained by adding the dashed part.} \label{fig:hierarchy}
\end{figure}
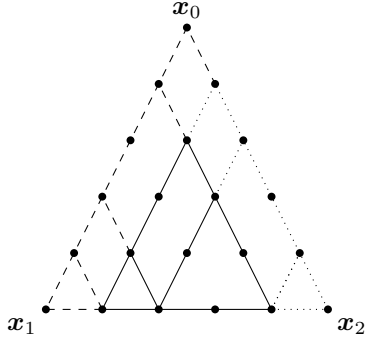

\end{remark}
Before proving unisolvence, we check that $\Gamma_r$ has the right cardinality. This is immediate from Remark \ref{rmk:hierar}.
\begin{lemma}
The set $\Gamma_r$ has cardinality $\lvert \Gamma_r \rvert$ equal to the dimension of $\mathcal{P}_{r-2}\Lambda^2(T)$, that is 
\[ \lvert \Gamma_r \rvert = \frac{r(r-1)}{2} .\]
\end{lemma}
\begin{proof}
We use induction on $r$. The result clearly holds for $r = 2$, see Figure \ref{fig:cells}. For $r>2$ the sets $\tau_{\boldsymbol{\zeta}_r}(\Gamma_{r-1})$ and $\Delta_r$ are disjoint, therefore the cardinality of $\Gamma_r$ is given by 
\begin{align*}
    \lvert \Gamma_r \rvert & = \lvert \Gamma_{r-1} \rvert + \lvert \Delta_r \rvert\\
                           & = \frac{(r-1)(r-2)}{2} + r - 1\\
                           & = \frac{r ( r - 1 ) }{ 2 }.
\end{align*} 
This concludes the proof. \end{proof}

To prove unisolvence of weights here defined we shall work as follows. Consider the sequence
\begin{equation} \label{eq:sequence} \mathcal{P}_r \Lambda^0 (T) \xrightarrow{\der} \mathcal{P}_{r-1} \Lambda^1 (T) \xrightarrow{\der} \mathcal{P}_{r-2} \Lambda^2 (T).
\end{equation}
The first and the last space are isomorphic under the action of the (smooth) Hodge star operator $ \star $ \cite{MarsdenBook}. This rather easy fact induces an interesting consequence, which consists in the fact that techniques adopted to prove unisolvence of the spaces at the extremity of \eqref{eq:sequence} are very close. On the contrary, unisolvence for the central space is obtained without direct computations but just relying on the structure of the sequence \eqref{eq:sequence} itself.

We are ready to prove the unisolvence of $\mathcal{F}$. 
\begin{theorem} \label{thm:main}
If the assumptions of Theorem 1 hold, then $\mathcal{F}$ is a unisolvent and minimal physical sysdofs. 
\end{theorem}
\begin{proof}
The minimality holds by construction for $k = 0$ and $k = 2$. For $k = 0$, unisolvence is just the standard Lagrange unisolvence on poised sets. For $k = 2$, let $\omega \in\mathcal{P}_{r-2}\Lambda^2(T)$ and assume that $w(\omega,s) = 0$ for each $s\in\mathcal{F}^2(T)$. Then, by linearity of the integral, it follows that 
\[ \int_{\tau_{\pmb{\xi}}(T)}\omega = 0,\quad\forall\pmb{\xi} \in \Gamma. \]
Then Theorem 1 implies $\omega = 0$. Finally, for $k = 1$, consider the following diagram: 
\begin{equation*}
    \small\begin{tikzcd}
0 \arrow[r] & \R \arrow[r, "\iota", hook] \arrow[d, "\Id"] & \mathcal{P}_r\Lambda^0(T) \arrow[r, "d"] \arrow[d, "\mathfrak{R}^0"] & \mathcal{P}_{r-1}\Lambda^1(T) \arrow[r, "d"] \arrow[d, "\mathfrak{R}^1"] & \mathcal{P}_{r-2}\Lambda^2(T) \arrow[r] \arrow[d, "\mathfrak{R}^2"] & 0 \\
0 \arrow[r] & \R \arrow[r, "\psi"]                         & C^0(\mathcal{F}^{\bullet}(T)) \arrow[r, "\delta"]                    & C^1(\mathcal{F}^{\bullet}(T)) \arrow[r, "\delta"]                        & C^2(\mathcal{F}^{\bullet}(T)) \arrow[r]                             & 0
\end{tikzcd}
\end{equation*}
We already know that the rows are exact and we have just showed that the maps $\mathfrak{R}^0$ and $\mathfrak{R}^2$ are isomorphisms. Then, by the Five Lemma (see section 2.1 of \cite{Hatcher}) it follows that also $\mathfrak{R}^1$ is an isomorphism. In particular minimality holds also for $k=1$. 
\end{proof}

The idea of proving the unisolvence of the intermediate space $k = 1$ using the Five Lemma appeared for the first time in \cite{ZAR}, but it was not exploited since a proof of the unisolvence for the case $k = 2$ was lacking. The problem with the physical sysdofs defined in \cite{ZAR} is that the $2$-cells cannot be written as differences of small $2$-simplices as in \eqref{eq : def case k = 2} and therefore Theorem 1 does not apply. 

We remark an interesting aspect. For $ k = 1 $, the set involved is the set of small simplices defined in \cite{nostropreprint}, which is a subset of that of \emph{small simplices} introduced by Bossavit in \cite{RapettiBossavit}. Interestingly, for $ k = 2 $ one does not find its $2$-dimensional counterpart, but the set $ \Sigma_r^2 (T) $ defined in \cite{ENUMATH}. To the authors' knowledge, this is the first construction in which those sets appears paired in such a natural fashion.

\section{Numerical tests} \label{sect:NumTest}

We offer a computational proof of unisolvece exploiting Lemma \ref{eq: alternative unisolvence}. We compute the conditioning number of the Vandermonde matrices of the sequence
$$ \mathcal{P}_r \Lambda^0 (T) \rightarrow \mathcal{P}_r \Lambda^1 (T) \rightarrow \mathcal{P}_r \Lambda^2 (T), $$
for $ r-2 = 1, \ldots, 4 .$ These quantities are reported in Table \ref{tab:condVk0} and confirm, up to the considered degree, the theoretical statement proved in Theorem \ref{thm:main}. The basis chosen for such computations is the monomial one, and barycentric coordinates offer a compact way to visualise it. In particular, when $k= 0$, it is defined as $ \boldsymbol{\lambda}^{\boldsymbol{\alpha}} $ with $ |\boldsymbol{\alpha} | = r$. When $ k = 1 $ it is defined as $ \boldsymbol{\lambda}^{\boldsymbol{\alpha}} \mathrm{d} x + \boldsymbol{\lambda}^{\boldsymbol{\beta}} \mathrm{d} y$ with $ |\boldsymbol{\alpha} | = r $ and $ |\boldsymbol{\beta} | = 0 $ and $ |\boldsymbol{\beta} | = r $ and $ |\boldsymbol{\alpha} | = 0 $. Finally, for $k = 2$, such a basis is $ \boldsymbol{\lambda}^{\boldsymbol{\alpha}} \mathrm{d} x \wedge \mathrm{d} y$, again with $ |\boldsymbol{\alpha} | = r $. Results for $ r = 1, \ldots, 6 $ are reported in Table \ref{tab:condVk0}. To improve conditioning numbers Bernstein bases or orthogonal polynomials shall be taken into account. However, unisolvence is clearly independent from the choice of the basis for $ \mathcal{P}_r \Lambda^k (T) $ and we thus leave this problem of optimisation to further investigations.

\begin{table}[h]
    \centering
    \begin{tabular}{c|ccc}
      $ r $ & $ k = 0 $ & $ k = 1$ & $ k = 2 $ \\
      \hline
      $ 1 $ & $ 3.7320 \times 10^0 $ & $ 4.4985 \times 10^0 $ & $ 3.1682 \times 10^1 $ \\
      $ 2 $ & $ 3.0969 \times 10^1 $ & $2.3281 \times 10^1 $ & $ 5.2130 \times 10^2 $ \\
      $ 3 $ & $ 3.1245 \times 10^2 $ & $ 8.6268 \times 10^1 $ & $ 9.3809 \times 10^3 $ \\
      $ 4 $ & $ 3.4290 \times 10^3 $ & $ 5.6267 \times 10^2 $  & $ 1.3525 \times 10^6 $ \\
      $ 5 $ & $ 3.9513 \times 10^4 $ & $2.9791 \times 10^3 $ & -- \\
      $ 6 $ & $ 4.7004 \times 10^5 $ & -- & -- \\
    \end{tabular}
    \caption{Conditioning number of the Vandermonde matrix for $k=0, 1, 2$.}
    \label{tab:condVk0}
\end{table}
 
\begin{remark}
We stress that Table \ref{tab:condVk0} shall be read \emph{diagonally}. In particular, when the degree for $ k = 0 $ is $ r $, the corresponding data for $ k = 1 $ and $ k = 2 $ are, respectively, those associated with $ r - 1 $ and $ r - 2 $.
\end{remark}

\subsection{Some interpolation tests}
In section \ref{sect:interpdef} we have defined how an interpolator can be defined using weights and we have briefly discussed its features. In particular, we showed that under the hypothesis of Theorem \ref{thm:main} such an operator is well defined and commutes with the exterior derivative. We now give an explicit meaning of this fact, using weights to interpolate a $ 0 $-form $ \omega $ and its differential $ \der \omega \in \Lambda^1 (T) $. For ease of the reader we deal with the standard $2$-simplex. This is not restrictive, since one may always reduce to this case by passing to barycentric coordinates. We thus consider a $ 0 $-form
$$ \omega = e^{x} \sin(\pi y),$$
whence
$$ d \omega = e^{x} \sin (\pi y) \der x + \pi e^x \cos(\pi y) \der y.$$
We interpolate by means of the interpolator \eqref{eq:interpolator} and study the convergence as $ r $ increases. The most informative norm for such a situation is the $0$-norm \cite{Harrison}, which is defined as
\begin{equation}
    \Vert \omega \Vert_0 \doteq \sup_{c \in \mathcal{C}_k (T)} \frac{1}{|c|_0}\left\vert \int_c \omega \right\vert,
\end{equation}
being $ |c|_0 $ the $k$-th volume of the $k$-simplex $ T $ and $ \mathcal{C}_k (T) \doteq \mathcal{C}_k(\mathcal{F}^{\bullet}(T))$ the set of all possible $k$-chains supported in $ T $.
Results are reported in Table \ref{tab:conv}, where a comparison with the corresponding points for $ k = 0 $ is included, and shown in Figure \ref{fig:conv}.

\begin{table}[h]
    \centering
    \begin{tabular}{c|cc}
    & $ k = 0 $ & $ k = 1 $ \\
    \hline 
    $ r $ & $ \Vert \omega - \Pi \omega \Vert_0 $ & $ \Vert \omega - \Pi \omega \Vert_0 $ \\
        \hline
      $ 1 $ & -- & $ 2.5334 $ \\
      $ 2 $ & $ 0.3377 \times 10^0 $ & $ 1.1224 $ \\
      $ 3 $ & $ 0.6967 \times 10^{-1} $ & $ 0.4292 $ \\
      $ 4 $ & $ 0.1792 \times 10^{-1} $ & $ 0.0782 $ \\
      $ 5 $ & $ 0.1600 \times 10^{-2} $ & $ 0.0171 $ \\
      $ 6 $ & $ 0.4314 \times 10^{-3} $ & --
    \end{tabular}
    \caption{Trend of $ \omega - \Pi \omega $ with respect to the $0$-norm for the $ 1 $-form $ \omega $ above defined and its potential. The $ 0 $-norm for of the function $ k = 0 $ is approximately $ 1.7319 $ whereas $ \Vert \omega \Vert_0 \sim 2.5334 $ for the case $ k = 1 $.} \label{tab:conv}
\end{table}

\begin{figure}[h]
    \centering
    \includegraphics[scale = .34]{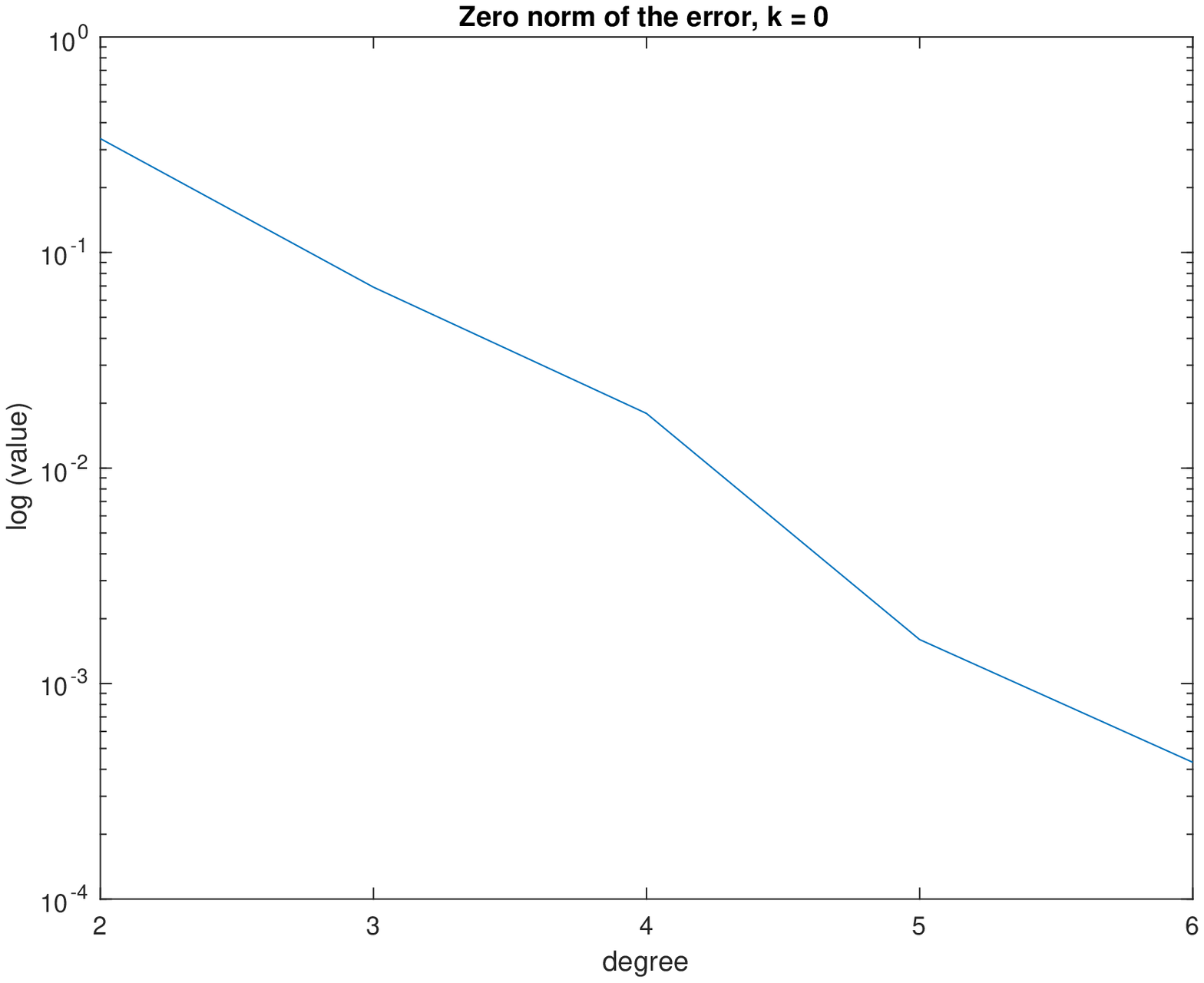}
    \quad 
    \includegraphics[scale = .34]{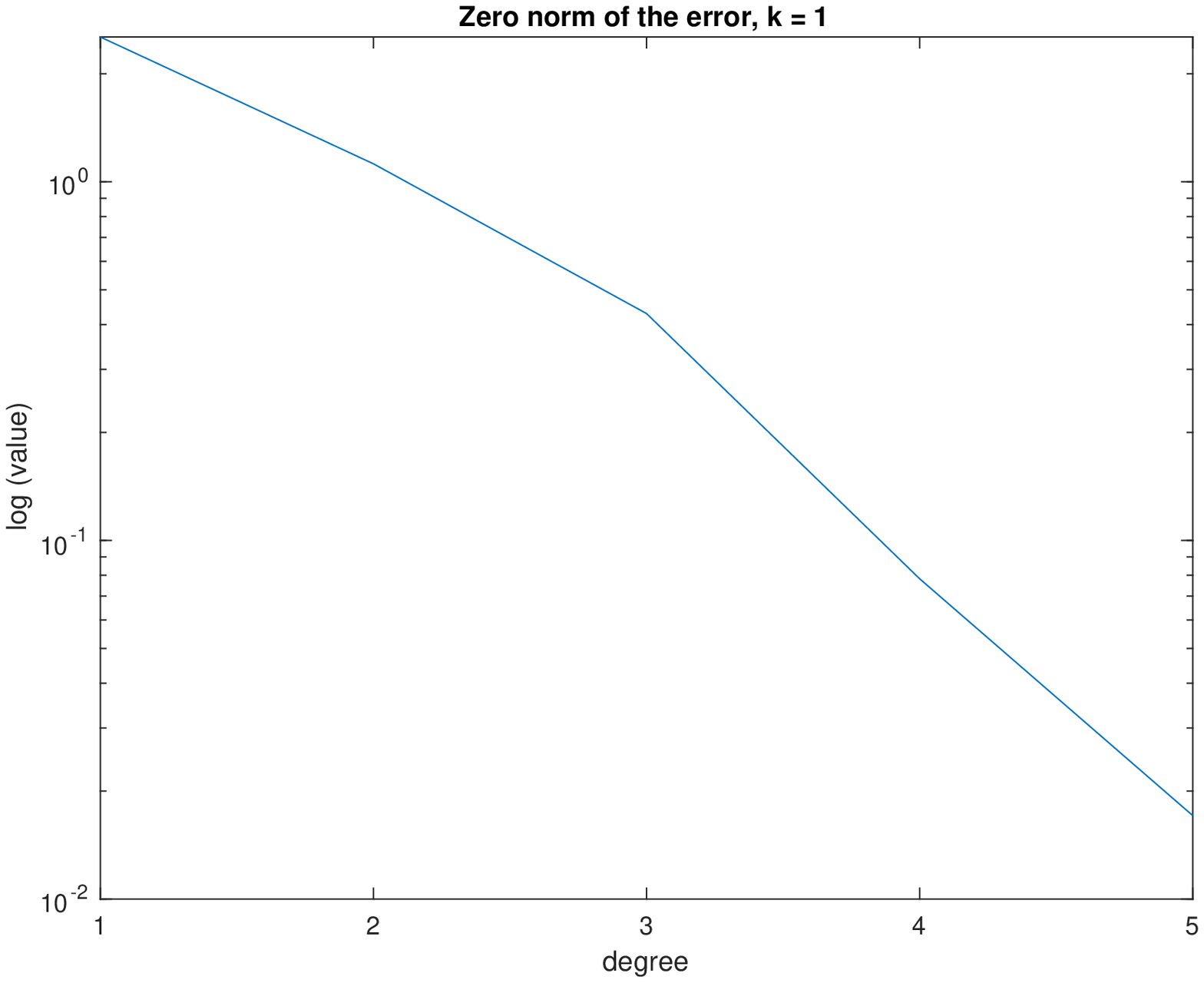}
    \caption{Plot of the convergence, a comparison for the nodal case $ k = 0 $ and the simplicial case $ k = 1 $ in semi-logarithmic scale. Left, the case for $ k = 0$ and right, that for $ k = 1 $. Notice the degree shift, explained by the sequence.}
    \label{fig:conv}
\end{figure}

\section{Conclusions and future directions} \label{sect:concl}
In this work we have proposed new physical degrees of freedom for the second family $\mathcal{P}_{r-k}\Lambda^k$ in the two dimensional case. We have proved rigorously their unisolvence and we have showed their effectiveness with an interpolation test. 

The three dimensional case is trickier. In principle one could use the same technique to construct unisolvent and minimal physical degrees of freedom for the case $k=3$, but unisolvence and minimality of the intermediate spaces in the sequence, that is $k = 1$ and $k=2$, will not follow trivially since the Five Lemma cannot be applied in this situation. This will be the object of future research.  

\bibliographystyle{plain}
\bibliography{bibliografia.bib}

\begin{thebibliography}{10}

\bibitem{MarsdenBook}
R.~Abraham, J.~E. Marsden, and T.~Ratiu.
\newblock {\em Manifolds, tensor analysis, and applications}, volume~75.
\newblock Springer Science \& Business Media, 2012.

\bibitem{ENUMATH}
A.~Alonso~Rodr{\'i}guez, L.~Bruni~Bruno, and F.~Rapetti.
\newblock Minimal sets of unisolvent weights for high order {W}hitney forms on
  simplices.
\newblock In {\em EnuMath 2019 procs.}, volume 139. Springer-Verlag, 2020.
\newblock LNCSE.

\bibitem{nostropreprint}
A.~Alonso~Rodr{\'i}guez, L.~Bruni~Bruno, and F.~Rapetti.
\newblock Towards nonuniform distributions of unisolvent weights for {W}hitney
  finite element spaces on simplices: the edge element case.
\newblock {\em Submitted}, 2021.

\bibitem{ABRICO}
A.~Alonso~Rodr{\'i}guez, L.~Bruni~Bruno, and F.~Rapetti.
\newblock Flexible weights for high order face based finite element
  interpolation.
\newblock {\em Submitted}, 2022.

\bibitem{GCGeneral}
A.~Apozyan, \'{G}. Avagyan, and G.~Ktryan.
\newblock On the {G}asca-{M}aeztu conjecture in {$\mathbb{R}^3$}.
\newblock {\em East J. Approx.}, 16(1):25--33, 2010.

\bibitem{FEEC}
D.~Arnold, R.~Falk, and R.~Winther.
\newblock Finite element exterior calculus, homological techniques, and
  applications.
\newblock {\em Acta Numerica}, 15:1 -- 155, 2006.

\bibitem{GCr4}
V.~Bayramyan, H.~Hakopian, and S.~Toroyan.
\newblock A simple proof of the {G}asca-{M}aeztu conjecture for {$n=4$}.
\newblock {\em Jaen J. Approx.}, 7(1):137--147, 2015.

\bibitem{BossBook}
A.~Bossavit.
\newblock {\em Computational electromagnetism}.
\newblock Electromagnetism. Academic Press, Inc., San Diego, CA, 1998.
\newblock Variational formulations, complementarity, edge elements.

\bibitem{BDM}
F.~Brezzi, J.~Douglas, and L.~D. Marini.
\newblock Two families of mixed finite elements for second order elliptic
  problems.
\newblock {\em Numer. Mathem.}, 47:217--235, 1985.

\bibitem{GCCarni}
J.~M. Carnicer and M.~Gasca.
\newblock A conjecture on multivariate polynomial interpolation.
\newblock {\em RACSAM. Rev. R. Acad. Cienc. Exactas F\'{\i}s. Nat. Ser. A
  Mat.}, 95(1):145--153, 2001.

\bibitem{ChristiansenTopics}
S.~Christiansen, H.~Munthe-Kaas, and B.~Owren.
\newblock Topics in structure preserving discretization.
\newblock {\em Acta Numerica}, 20:1 -- 119, 2011.

\bibitem{ChristiansenConstruction}
S.~H. Christiansen.
\newblock A construction of spaces of compatible differential forms on cellular
  complexes.
\newblock {\em Mathematical Models and Methods in Applied Sciences},
  18(5):739--757, 2008.

\bibitem{christiansen2009Foundations}
S.~H. Christiansen.
\newblock Foundations of finite element methods for wave equations of {M}axwell
  type.
\newblock In {\em Applied wave mathematics}, pages 335--393. Springer, 2009.

\bibitem{ChristiansenRapetti}
S.~H. Christiansen and F.~Rapetti.
\newblock On high order finite element spaces of differential forms.
\newblock {\em Math. Comput.}, 85:517--548, 2016.

\bibitem{ChungYao}
K.~C. Chung and T.~H. Yao.
\newblock On lattices admitting unique {L}agrange interpolations.
\newblock {\em SIAM J. Numer. Anal.}, 14(4):735--743, 1977.

\bibitem{Ciarlet}
P.~G. Ciarlet.
\newblock {\em The finite element method for elliptic problems}.
\newblock North-Holland Publishing Co, 1978.

\bibitem{deBoor}
C.~de~Boor.
\newblock Multivariate polynomial interpolation: conjectures concerning
  {GC}-sets.
\newblock {\em Numer. Algorithms}, 45(1-4):113--125, 2007.

\bibitem{GCGM}
M.~Gasca and J.~I. Maeztu.
\newblock On {L}agrange and {H}ermite interpolation in $ \mathbb{R}^k $.
\newblock {\em Numer. Math}, 39:1--14, 1982.

\bibitem{Gopa}
J.~Gopalakrishnan, L.~E. Garc\'{\i}a-Castillo, and L.~F. Demkowicz.
\newblock N\'{e}d\'{e}lec spaces in affine coordinates.
\newblock {\em Comput. Math. Appl.}, 49(7-8):1285--1294, 2005.

\bibitem{GCr5}
H.~Hakopian, K.~Jetter, and G.~Zimmermann.
\newblock The {G}asca-{M}aeztu conjecture for {$n=5$}.
\newblock {\em Numer. Math.}, 127(4):685--713, 2014.

\bibitem{Harrison}
J.~Harrison.
\newblock Continuity of the integral as a function of the domain.
\newblock {\em Journal of Geometric Analysis}, 8(5):769--795, 1998.

\bibitem{Hatcher}
A.~Hatcher.
\newblock {\em {Algebraic topology}}.
\newblock Cambridge Univ. Press, Cambridge, 2000.

\bibitem{Hipt}
R.~Hiptmair.
\newblock Canonical construction of finite elements.
\newblock {\em Math. Comp.}, 68(228):1325--1346, 1999.

\bibitem{Lohi1}
L.~Kettunen, J.~Lohi, Jukka R\"{a}bin\"{a}, Sanna M\"{o}nk\"{o}l\"{a}, and
  Tuomo Rossi.
\newblock Generalized finite difference schemes with higher order {W}hitney
  forms.
\newblock {\em ESAIM Math. Model. Numer. Anal.}, 55(4):1439--1459, 2021.

\bibitem{Lee}
J.~M. Lee.
\newblock {\em Introduction to smooth manifolds}, volume 218 of {\em Graduate
  Texts in Mathematics}.
\newblock Springer, second edition, 2012.

\bibitem{Lohi2}
J.~Lohi and L.~Kettunen.
\newblock Whitney forms and their extensions.
\newblock {\em J. Comput. Appl. Math.}, 393:Paper No. 113520, 19, 2021.

\bibitem{NedelecFirst}
J.-C. N\'{e}d\'{e}lec.
\newblock Mixed finite elements in $\mathbb{R}^3$.
\newblock {\em Numer. Mathem.}, 35:315--342, 1980.

\bibitem{NedelecSecond}
J.-C. N\'{e}d\'{e}lec.
\newblock A new family of mixed finite elements in $\mathbb{R}^3$.
\newblock {\em Numer. Mathem.}, 50:57–81, 1986.

\bibitem{RapettiBossavit}
F.~Rapetti and A.~Bossavit.
\newblock Whitney forms of higher degree.
\newblock {\em SIAM J. Numer. Anal.}, 47:2369--2386, 2009.

\bibitem{Whitney}
H.~Whitney.
\newblock {\em Geometric Integration Theory}.
\newblock Princeton University Press, 1957.

\bibitem{ZAR}
E.~Zampa, A.~Alonso Rodriguez, and F.~Rapetti.
\newblock Using the f.e.s. framework to derive new physical degrees of freedom
  for {N}édélec spaces in two dimensions.
\newblock {\em Submitted}, 2021.

\end{thebibliography}

\end{document}